\documentclass[11pt]{article}
\usepackage{mathrsfs} 
\usepackage[margin=1.25in]{geometry}
\usepackage{ifthen}
\usepackage{graphicx}

\usepackage{epsfig,fancyref}
\usepackage{dsfont}
\usepackage{amsfonts,dsfont,amssymb,amsthm,stmaryrd,bbm}
\usepackage[cmex10]{amsmath} 
\usepackage{mathtools}

\usepackage{hyperref}
\hypersetup{colorlinks=true,pdftitle=""}

\newtheorem{conj}{Conjecture}[section]
\newtheorem{thm}{Theorem}[section]

\newtheorem{lem}[conj]{Lemma}
\newtheorem{prop}[conj]{Proposition}

\newtheorem{ques}[conj]{Question}

\newtheorem{defn}[conj]{Definition}
\newtheorem{cor}[conj]{Corollary}

\newcommand\independent{\protect\mathpalette{\protect\independent}{\perp}} 
\def\independent#1#2{\mathrel{\rlap{$#1#2$}\mkern2mu{#1#2}}}

\def\phi{\varphi}

\def\bee{\begin{eqnarray*}}
	\def\ene{\end{eqnarray*}}

\begin{document}
		\title{Rearrangement and Pr\'ekopa-Leindler type inequalities}
		\author{James Melbourne}
		\date{May 10, 2019}
		
		\maketitle
		
		\begin{abstract}
		    We investigate the interactions of functional rearrangements with Pr\'ekopa-Leindler type inequalities.  It is shown that that certain set theoretic rearrangement inequalities can be lifted to functional analogs, thus demonstrating that several important integral inequalities tighten on functional rearrangement  about ``isoperimetric" sets with respect to a relevant measure.  Applications to the Borell-Brascamp-Lieb, Borell-Ehrhard, and the recent polar Pr\'ekopa-Leindler inequalities are demonstrated.  It is also proven that an integrated form of the Gaussian log-Sobolev inequality sharpens on rearrangement.
		\end{abstract}
\section{Introduction}

The Pr\'ekopa-Leindler inequality (PLI) stated below has become a useful tool in the study of log-concave distributions in probability and statistics, particularly in high dimension, and a point of interest and unification between probabilists and convex geometers.
\begin{thm}[Pr\'ekopa-Leindler] \label{thm: PLI}
    For $f,g :\mathbb{R}^d \to [0,\infty)$ Borel measurable and  $t \in (0,1)$, define 
        \begin{align*}
            f \square g(z) \coloneqq \sup_{(1-t)x + t y =z} f^{1-t}(x) g^t(y)
        \end{align*}
    then
        \begin{align*}
            \int_{\mathbb{R}^d} f \square g (z) dz \geq \left( \int_{\mathbb{R}^d} g(z) dz \right)^{1-t} \left( \int_{\mathbb{R}^d} h(z) dz \right)^t.
        \end{align*}
\end{thm}

The inequality can be motivated from a convex geometric perspective as a functional generalization of the dimension free statement of the Brunn-Minkowski inequality (BMI), which we recall as the fact that for $A,B$ compact in $\mathbb{R}^d$ and  $|\cdot|_d$ the $d$-dimensional Lebesgue volume,
\[
    |(1-t) A + t B |_d \geq |A|^{1-t}_d |B|_d^{t}.
\]
Indeed by taking $f = \mathbbm{1}_A$, and $g = \mathbbm{1}_B$, we have $f\square g = \mathbbm{1}_{(1-t)A+tB}$. PLI implies that integration preserves the inequality and the result follows.  

The BMI has an elegant qualitative formulation;
the volume of sum-sets decreases on spherical symmetrization.  More explicitly, if $A$ and $B$ are compact sets, with $A^*$ and $B^*$ Euclidean balls satisfying $|A^*|_d = |A|_d$, $|B^*|_d =|B|_d$, then
\begin{align} \label{eq: qualitative BMI}
    |A + B|_d \geq |A^* + B^*|_d.
\end{align}

Our first main result (Theorem \ref{thm: PLI rearrangement}) contains a functional generalization of \eqref{eq: qualitative BMI}.  We will show PLI  ``sharpens'' on rearrangement in the sense that 
\begin{align} \label{eq: spherical rearrangement example}
    \int f \square g \geq \int f^* \square g^*
\end{align}
where $*$ denotes a functional rearrangement to be defined below.  In fact we will prove that for $\psi$ increasing, 
\begin{align} \label{eq: spherical rearrangement example increasing function}
    \int \psi(f \square g) \geq \int \psi(f^* \square g^*).
\end{align}
Our methods are reasonably general and Theorem \ref{thm: Main Theorem} will give a class of set theoretic inequalities that admit functional generalization in the sense of \eqref{eq: spherical rearrangement example increasing function}.  As a consequence we will show that analogs of \eqref{eq: spherical rearrangement example increasing function} can be given to sharpen not only the PLI, but the Borell-Brascamp-Lieb inequalities \cite{Bor75a,BL76a}, the Borell-Ehrhard inequality in the Gaussian setting \cite{Bor08, Ehr83}, and a recent Polar Pr\'ekopa-Leindler \cite{artstein2017polar}.

These results can also be motivated from an information theoretic perspective, where the BMI can be considered a R\'enyi entropy power inequality. There has been considerable recent work (see \cite{BM16, BC14, BC15:1, Li17, LMM2019:isit:1,LMM2019:isit:2, RS16}) developing R\'enyi entropy \cite{Ren61} generalizations of the classical entropy power inequality (EPI) of Shannon-Stam \cite{Sha48, Sta69}.  One should compare the sharpening of PLI here to \cite{WM14}, where Madiman and Wang show that while spherically symmetric decreasing rearrangements of random variables preserve their R\'enyi entropy, they decrease the R\'enyi entropy of independent sums of random variables. 
One application of the rearrangement result in information theory is the reduction of R\'enyi generalizations of the EPI to the spherically symmetric case, see for example \cite{marsiglietti2018renyi} where the Madiman-Wang result is used to sharpen the R\'enyi EPI put forth in \cite{marsiglietti2018entropy}.  See \cite{MMX17:2} to find an extension and application of \cite{WM14} for the $\infty$-R\'enyi entropy.
It should be mentioned that the connections between BMI and entropy power inequalities are not new.   The analogy between the two inequalities was first observed in \cite{CC84}, and a unified proof was given in \cite{DCT91} drawing on the work of \cite{Bec75, BL76b,  Lie78}.  The reader is directed to \cite{MMX17:1} where a further development of R\'enyi entropy power inequalities and their connections to convex geometry are given.

In the Gaussian case, the strict convexity of the potential gives a result stronger than PLI, and we are able to adapt the rearrangement ideas to approach the Gaussian log-Sobolev inequality.  We show in Theorem \ref{thm: Integrated log-Sobolev} that for the Gaussian measure, the ``integrated'' log-Sobolev inequality derived from PLI by Bobkov and Ledoux \cite{BL09} sharpens on half-space rearrangement.


An alternative motivation for this investigation is the Brascamp-Lieb-Barthe  inequalities relationship to the Brascamp-Lieb-Luttinger rearrangement inequalities \cite{BLL74}. The Brascamp-Lieb inequality \cite{BL76a} enjoys the Brascamp-Lieb-Luttinger inequality as a rearrangement analog.  In \cite{Bar98b} Barthe used an optimal transport argument to prove Brascamp-Lieb and simultaneously demonstrated a dual inequality that includes PLI as a special case.  It is natural to ask for a rearrangement inequality analog of Barthe's result, to provide a dual to the Brascamp-Lieb-Luttinger rearrangement inequality. This work represents a confirmation of such an inequality in the special case corresponding to PLI.  

The paper is organized in the following manner; in Section \ref{sec:Preliminaries} we will give defintions and background on a notion of rearrangement. In Section \ref{sec: PLI rearrange} we give a rearrangement inequality for PLI,  before giving a general version in Section \ref{sec: General rearrange}.  In Section \ref{sec: Applications} we give applications of the theorem derived in Section \ref{sec: General rearrange} to special cases. 
In Section \ref{sec: Gaussian log-Sobolev} we give a sharpening of an integrated Gaussian log-Sobolev inequality, via half-space rearrangement.  Finally, in Section \ref{sec: Barthe Rearrangement Connection} we discuss connections with the work of Barthe and Brascamp-Lieb-Luttinger closing with an open problem.

\section{Preliminaries} \label{sec:Preliminaries}
For a set $A$, will use the notation $\mathbbm{1}_A$ to denote the indicator function of $A$, taking the value $1$ on $A$, and $0$ elsewhere. For $x \in \mathbb{R}^d$, $|x|$ will denote the usual Euclidean norm.  We use $\mathbb{Q}_+$ to denote the non-negative rational numbers.
We use $\gamma_d$ to denote both the standard Gaussian measure on $\mathbb{R}^d$ and its density function
\[
    \gamma_d(x) = \frac{e^{-|x|^2/2}}{(2 \pi)^{\frac d 2}}.
\]
When $d=1$, and there is no risk of confusion, we will omit the subscript and write $\gamma$.
We denote the Gaussian distribution function
\[
    \Phi(x) = \int_{-\infty}^x \gamma(y) dy
\]
and its inverse $\Phi^{-1}$.

\subsection{Spherically symmetric decreasing rearrangements}
Given a nonempty measurable set $A \subseteq \mathbb{R}^d$ we define its spherically symmetric rearrangement $A^*$ to be the origin centered ball of equal volume,
\begin{align*}
	A^* \coloneqq \left\{x: |x| < \left(|A|_d /\omega_d \right)^{\frac 1 d} \right\},
\end{align*}
where $\omega_d$ is the volume of the $d$-dimensional unit ball, with the understanding that $A^* = \varnothing$ in the case that $|A|_d =0$ and $A^* = \mathbb{R}^d$ when $|A|_d = \infty$.

We can extend this notion of symmetrization to functions via the layer-cake decomposition of a non-negative function $f$, 
\begin{align*}
	f(x) =  \int_0^{f(x)} 1 dt = \int_0^\infty \mathbbm{1}_{ \{y : f(y) > t \}} (x) dt.
\end{align*}

\begin{defn}
	For a measurable non-negative function $f$ define its decreasing symmetric rearrangement $f^*$ by
	\begin{equation} \label{eq: non increasing symmetric rearrangement defn}
		f^*(x) \coloneqq \int_0^\infty \mathbbm{1}_{ \{ y: f(y) > t \}^*}(x) dt.
	\end{equation}
\end{defn}

Note that {\it decreasing} is used here in the non-strict sense, synonomous with non-increasing.

\begin{prop}
$f^*$ is characterized by the equality 
\begin{equation} \label{eq: symmetric rearrangement characterization}
	\{f^* > \lambda \} = \{ f> \lambda \}^*.
\end{equation}
\end{prop}

The proof will be given in greater generality in the following section.

\begin{cor}
	$f^*$ is lower semi-continuous, spherically symmetric and non-increasing in the sense that $|x| \leq |y|$ implies $f^*(x) \geq f^*(y)$.
\end{cor}
\begin{proof}
	$f^*$ has open super level sets by equation \eqref{eq: symmetric rearrangement characterization}, and is thus lower semi-continuous.  To prove non-increasingness observe that using the characterization above $f^*(y) > \lambda$ iff $y \in \{ f > \lambda \}^*$ which implies by $|x| \leq |y|$ that $x \in \{f > \lambda\}^*$, and thus $f^*(x) > \lambda$.  Applying this to $\lambda_n$ increasing to $f^*(y)$ yields our result.  Observe that this implies spherical symmetry, by applying preceding argument in the opposite direction $f(x) = f(y)$ when $|x| = |y|$. 
\end{proof}

\subsection{More general rearrangements}

\begin{defn} \label{def: rearrangement} \normalfont
For Polish measure spaces $(M, \mu)$ and $(N, \alpha)$, with Borel $\sigma$-algebra, we will call a set map from the Borel $\sigma$-algebra of $M$ to the Borel $\sigma$-algebra of $N$
a {\it rearrangement} when it satisfies the following,
\begin{enumerate}
    \item \label{item: rearrangements are open}
    $*(A)$ is an open set satisfying $\alpha(*(A)) = \mu(A)$
    \item \label{item: non-increasingness}
    $\mu(A) \leq \mu(B)$ implies $*(A) \subseteq *(B)$
    \item \label{item: continuity of rearrangement}
    For a sequence $A_i \subseteq A_{i+1}$, $*( \cup_{i=1}^\infty A_i) = \cup_{i=1}^\infty *(A_i).$
\end{enumerate}
\end{defn}
Notice that in \ref{item: continuity of rearrangement}, $\cup_j *(A_j) \subseteq *( \cup_j A_j)$ holds from \ref{item: non-increasingness}, so the assumption is only $\cup_j *(A_j) \supseteq *( \cup_j A_j)$.
For brevity of notation we write $A^* = *(A)$, and note the following extension to functions.  
\begin{defn}
    For a rearrangement $*$ and Borel measurable $f:M \to [0,\infty)$ 
    define $f^*: N \to [0,\infty)$,
    \begin{align*}
        f^*(x) \coloneqq \int_0^\infty \mathbbm{1}_{\{ f > t \}^*}(x) dt.
    \end{align*}
\end{defn}

Rearrangement is in general non-linear, however we do have linear behavior in the following special case.
\begin{lem} \label{lem: simple function rearrangement}
For a simple function $s$, expressed as 
$
    s = \sum_{i=1}^n a_i \mathbbm{1}_{A_i}
$
with $a_i >0$ and $A_i \subsetneq A_{i-1}$,
\begin{align*}
    s^* = \sum_{i=1}^n a_i \mathbbm{1}_{A_i^*}.
\end{align*}
\end{lem}

\begin{proof}
Let us give more explicit formulas for both quantities.
    $$\sum_{i=1}^n a_i \mathbbm{1}_{A_i^*}(z) = \sum_{i=1}^{m_z} a_i$$
    where $m_z = \max \{ i : z \in A_i^*\}$, and the formula
  $$
        s^*(z) = \sup \{ t : z \in \{ s > t \}^* \},
    $$ which holds not just for simple functions but general $f$.
    If $z \in A^*_{m_z}$ with $m_z$ maximal, then for $t < \sum_{i=1}^{m_z} a_i$, $A_{m_z} \subseteq \{ s > t\}$, which in turn gives $A_{m_z}^* \subseteq \{s > t \}^*$.  Thus $z \in \{ s >t \}^*$ for all $t < \sum_{i=1}^{m_z} a_i$ and we have
    \begin{align*}
        s^*(z)
            =
            \sup_t \{ z \in \{s > t \}^* \}
            \geq
                \sum_{i=1}^{m_z} a_i
                =
                    \sum_{i=1}^n a_i \mathbbm{1}_{A_i^*}(z).
    \end{align*}
    For the reverse inequality, assume $s^*(z)>0$ (else there is nothing to prove) and take $t$ such that $z \in \{ s > t \}^*$.  Since $\{ s >t \} = A_{k_t}$ where $k_t = \min \{j : \sum_{i=1}^j a_i > t \}$, we have $\{ s >t \}^* = A_{k_t}^*$.  This implies that $\sum_{i=1} a_i \mathbbm{1}_{A_i^*}(z) \geq \sum_{i=1}^{k_t} a_i > t$.  Taking the supremum in $t$,
    \[
        \sum_{i=1}^n a_i \mathbbm{1}_{A_i^*}(z) \geq s^*(z).
    \]
\end{proof}

\begin{prop} \label{prop: characterization of rearrangement}
$f^*$ is characterized by the equality
    \begin{align} \label{eq: characterizing equation for f^*}
    \{f^* > \lambda \} = \{f > \lambda \}^*.
    \end{align}  In particular $f^*$ is lower semi-continuous, 
    and equi-measureable with $f$ in that $\mu \{ f > \lambda \} = \alpha \{ f^* > \lambda \}$.
\end{prop}


\begin{proof}
    First we prove the equality \eqref{eq: characterizing equation for f^*}.  Since $f^*(x) > \lambda$ implies $\int_0^\infty \mathbbm{1}_{\{f > t \}^*}(x) dt > \lambda$, which in turn, by the monotonicity of $\mathbbm{1}_{\{ f > t \}^*}$ implies the existence of $t > \lambda$ such that $x \in \{ f> t\}^*$.  From this it follows that
    \begin{align*}
        \{ f^* > \lambda \} \subseteq \{ f > \lambda \}^*.
    \end{align*}
    For the converse, first assume that $f = s$ is a simple function, expressed as
    \[
        s = \sum_{i=1}^n a_i \mathbbm{1}_{A_i}
    \]
    with $a_i >0$ and $A_{i} \subsetneq A_{i-1}$.  By Lemma \ref{lem: simple function rearrangement}
    \[
    s^* = \sum_{i=1}^n a_i \mathbbm{1}_{A_i^*}.
    \]
    Since $\{ s > \lambda \} = A_k$ where $k = \min \{j : \sum_{i=1}^j a_i > \lambda \}$, $z \in \{ s > \lambda \}^* = A_k^*$ implies $s^*(z) = \sum_{i=1}^n a_i \mathbbm{1}_{A_i^*}(z) \geq \sum_{i=1}^k a_i > \lambda$.  Thus $\{ s > \lambda \}^* \subseteq \{ s^* > \lambda \}$ holds for simple functions.
    Now take $s_n$ to be a sequence of increasing simple functions approximating $f$ pointwise, and uniformly on sets where $f$ is bounded.  Then
    \begin{align*}
        \{ f > \lambda \}^* 
            =
                \left( \bigcup_{n=1}^\infty \{ s_n > \lambda \} \right)^*
            =
                \bigcup_{n=1}^\infty \{s_n > \lambda \}^*
            =
                \bigcup_{n=1}^\infty \{ s_n^* > \lambda \}.
    \end{align*}
    where the first equality is from the assumption of increasingness of the simple functions, the second is from the Definition \ref{def: rearrangement} item \eqref{item: continuity of rearrangement}, and the third follows from the characterization just proven for simple functions.
Since $f_1 \leq f_2$, implies $f_1^* \leq f_2^*$ it follows that $\cup \{ s_n^* > \lambda \} \subseteq \{f^* > \lambda \}$, so that $\{ f > \lambda \}^* \subseteq \{f^* > \lambda\}$.

If $g$ is another function satisfying $\{ g > \lambda\} = \{ f > \lambda \}^*$ for all $\lambda$, then
\begin{align*}
    g(z) 
        =
            \int_0^\infty \mathbbm{1}_{\{ g > \lambda \}} d \lambda 
        =
            \int_0^\infty \mathbbm{1}_{\{ f > \lambda \}^*} d \lambda 
        =
            \int_0^\infty \mathbbm{1}_{\{ f^* > \lambda \}} d \lambda 
        =
            f^*(z).
\end{align*}

The fact that $f$ is lower semi-continuous follows from item \eqref{item: rearrangements are open} of our definition, that $A^*$ is open.  Equimeasurability is given by $\alpha \{ f^* > \lambda\} = \alpha \{ f > \lambda \}^* = \mu \{ f > \lambda \}$.
\end{proof}

\begin{prop} \label{prop: convex set rearrangement holds}
For an open convex set $K \subseteq \mathbb{R}^d$ with closure containing the origin. The set map $*_K$ defined by $$A^{*_K} \coloneqq \left(\frac{|A|_d}{|K|_d} \right)^{\frac 1 d} K,$$ is a rearrangement with $(M,\mu) = (N, \alpha) = (\mathbb{R}^d, |\cdot|_d)$.
\end{prop}

\begin{proof}
It is immediate that $A^{*_K}$ is open and the homogeneity of the Lebesgue measure ensures that $|A^{*_K}|_d = |A|_d$, hence \eqref{item: rearrangements are open} follows.  To prove \eqref{item: non-increasingness}, note that for $0 < |A| \leq |B|$, by the definition of $*_K$,  $A^{*_K} = tK$ and $B^{*_K} = sK$ for some $0 < t \leq s$.
Suppose that $x = t k$ for $k \in K$ and $k_n$ a sequence in $K$ converging to $0$.  Then $$x = s\left( \frac t s \left(k - \left(\frac s t - 1 \right) k_n \right) + \left(1 - \frac t s \right) k_n\right).$$ 
By $K$ open, $k- (\frac s t - 1) k_n$ belongs to $K$ for large $n$, and when this holds, by convexity $( \frac t s (k - (\frac s t - 1 )k_n) + (1 - \frac t s) k_n) \in K$.  It follows that $x \in s K$ and hence $A^{*_K} \subseteq B^{*_K}$.  The continuity condition in \eqref{item: continuity of rearrangement} holds, since both sets are origin symmetric balls of the same volume.
\end{proof}

Observe that the qualitative statement of Brunn-Minkowski \eqref{eq: qualitative BMI}, for Borel $A,B$
\begin{align} \label{eq: K convex set BMI}
    |A+B|_d \geq |A^{*_K} + B^{*_K}|_d,
\end{align}
is preserved.  In the following section we will extend this qualitative result to the functional setting.


\begin{prop} \label{prop: half space rearrangement holds}
    For a fixed coordinate $i$, the set function $*$ defined on a Polish space $M$ with probability measure $\mu$ and $(N,\alpha) = (\mathbb{R}^d, \gamma_d)$ by
    \[
        A^* = \{ x : x_i < \Phi^{-1}(\mu(A)) \}
    \]
    is a rearrangement.
\end{prop}
\begin{proof}
$A^*$ is open by definition, and $\gamma_d( A^*) = \Phi(\Phi^{-1}(\mu(A))) = \mu(A).$  Conditions \eqref{item: non-increasingness} and \eqref{item: continuity of rearrangement} follow from the monotonicity and continuity of $\Phi$.
\end{proof}

\section{Rearrangement and Pr\'ekopa-Leindler} \label{sec: PLI rearrange}
We begin with a special case of a more general result to build some intuition for the abstractions to follow.  For $f, g: \mathbb{R}^d \to [0,\infty)$ and $t \in [0,1]$ recall
\begin{align} \label{eq: PLI square operation}
    f \square g(z) = \sup_{(1-t) x + ty = z} f^{1-t}(x) g^t(y).
\end{align}
\begin{thm} \label{thm: PLI rearrangement}
    For $f,g : \mathbb{R}^d \to [0,\infty)$ Borel, $t \in (0,1)$, and $*$ denoting a rearrangement to a fixed open convex set with closure containing the origin, 
    \begin{align} \label{eq: qualitative PLI}
    \int_{\mathbb{R}^d} f \square g(z) dz \geq \int_{\mathbb{R}^d} f^* \square g^*(z) dz \geq  \left( \int f dz \right)^{1-t} \left( \int g dz \right)^t.
    \end{align}
    What is more, when $\psi$ is a non-negative and non-decreasing function
    \begin{align} \label{eq: qualitative PLI extended to increasing functions}
        \int_{\mathbb{R}^d} \psi(f \square g) (z) dz \geq \int_{\mathbb{R}^d} \psi( f^* \square g^* )(z) dz.
    \end{align}
\end{thm}

    The universal measurability of $f \square g$ will follow from the proof, which gives the universal measurability of $\psi( f \square g)$ as a consequence.

\begin{proof}
  For $\lambda \in (0,\infty)$, define 
        \begin{align} \label{eq:Set defintion for indexing}
        S_0 = S_0(\lambda) = \{ s \in \mathbb{Q}^2_+ : s_1^{1-t} s_2^t > \lambda \}.
        \end{align}
    Observe,
    \begin{align} \label{eq:characterizing formula for super level sets}
        \{ f \square g > \lambda \} = \bigcup_{s \in S_0(\lambda)} (1-t) \{ f > s_1\} + t \{ g > s_2\}.
    \end{align}
    Indeed, it is routine to check that $z \in \cup_{s \in S_0} (1-t) \{ f > s_1 \} + t \{ g > s_2\}$ implies $f \square g(z) > \lambda$. Conversely, if $f\square g (z) > \lambda $, then there exists a pair of $x$ and $y$ such that $(1-t)x + t y =z$ and $f^{1-t}(x)g^t(y) > \lambda$.  By the continuity of the map $(u,v) \mapsto u^{1-t}v^t$, there exists $(s_1,s_2)$ rational satisfying $s_1< f(x)$, $s_2 < g(y)$, and $s_1^{1-t}s_2^t > \lambda$, which proves the claim.
    
    Let us remark, that the sum of Borel sets is universally measurable\footnote{This follows from the fact that Borel sets are analytic, see \cite{kechris2012classical}, and analytic sets are closed under summation and universally measurable.}, and hence $\{ f \square g > \lambda \}$ is as well.  This shows we are well justified in our notation $\int_{\mathbb{R}^d} f \square g (z) dz$.  By Brunn-Minkowski and the characterizing property of rearrangements on super level sets
    \begin{align} \label{eq: BMI on superlevel sets}
        |(1-t) \{ f > s_1 \} + t \{ g > s_2\}|
            &\geq
                |(1-t) \{ f > s_1 \}^* + t \{ g > s_2\}^*|
                    \\
            &=
                |(1-t) \{ f^* > s_1 \} + t \{ g^* > s_2\}|.
    \end{align}
    Now applying \eqref{eq:characterizing formula for super level sets} to $f^* \square g^*$ and observing that,
    \begin{align*}
        (1-t) \{ f^* > s_1 \} + t \{ g^* > s_2\}
    \end{align*}
    is an origin centered ball in $\mathbb{R}^d$ for every $s \in S_0(\lambda)$, we see that 
    \begin{align*}
       | \{ f^* \square g^* > \lambda \}|
            &=
                \left| \bigcup_{s \in S_0(\lambda)} (1-t)\{ f^* > s_1 \} + t \{ g^* > s_2 \} \right|
                    \\
            &= \sup_{s \in S_0} \left| (1-t)\{ f^* > s_1 \} + t \{ g^* > s_2 \} \right|.
    \end{align*}
    But using \eqref{eq: BMI on superlevel sets}, obviously
        \begin{align*}
            \left| (1-t)\{ f^* > s_1 \} + t \{ g^* > s_2 \}  \right|  \leq \left| \bigcup_{s \in S_0(\lambda)} (1-t) \{ f > s_1\} + t \{ g > s_2\} \right|
        \end{align*}
        and thus it follows that
        \begin{align} \label{eq: measure dominance on sup convo}
            |\{ f \square g > \lambda \}| \geq | \{ f^* \square g^* > \lambda \}|.
        \end{align}
    Using the layer-cake decomposition of the integral
    \begin{align*}
        \int_{\mathbb{R}^d} \psi(f\square g)(z) dz = \int_0^\infty |\{ \psi(f \square g ) > t \}| d t.
    \end{align*}
    Notice that by the non-decreasingness, $\psi^{-1}(\lambda, \infty)$ is an interval of the form $[x,\infty)$ or $(x,\infty)$ for a non-negative $x$, and from this, we can use \eqref{eq: measure dominance on sup convo} (and continuity of measure if the interval is closed) we obtain \eqref{eq: qualitative PLI extended to increasing functions}.  To recover \eqref{eq: qualitative PLI}, note that the first inequality follows from setting $\psi(x) = x$, while the second is the application of PLI to $f^*$ and $g^*$ combined with the equimeasurability of the rearrangements ensuring $\int f^* = \int f$ and $\int g^* = \int g$.
\end{proof}

\section{Functional lifting of rearrangements} \label{sec: General rearrange}

In this section we show that in a general setting, certain set theoretic rearrangement inequalities can be extended to functional analogs, extending the rearrangement inequality proven for PLI in the previous section to more general operations than $\square$ in \eqref{eq: PLI square operation}.    
Let us make precise the set theoretic rearrangement inequality we will generalize.  
\begin{defn}
Let $m:M^n \to M$ and $\eta: N^n \to N$ be such that $m(A_1, \dots, A_n) = \{x = m(a_1, \dots, a_n) : a_i \in A_i \}$ and $\eta(B_1, \dots, B_n) = \{ y = \eta(b_1, \dots, b_n) : b_i \in B_i \}$ are universally measurable for $A_i$ and $B_j$ Borel.  Suppose further that $\{\eta(A_1^*, \dots, A_n^*)\}_A$ indexed on $n$-tuples of Borel sets is totally ordered in the sense that for any Borel $A_1, \dots, A_n$ and $A'_1, \dots, A'_n$ we have either
$
    \eta(A_1^*, \dots, A_n^*) \subseteq \eta({A'}_1^*, \dots, {A'}_n^*)
\mbox{ or } 
    \eta(A_1^*, \dots, A_n^*) \supseteq \eta({A'}_1^*, \dots, {A'}_n^*)
$
we say that $*$ satisfies a set theoretic rearrangement inequality when the following holds
\begin{align*}
    \mu( m(A_1, \dots, A_n))  \geq \alpha( \eta (A_1^*, \dots, A_n^*)).
\end{align*}
\end{defn}

We will focus on two main examples, the rearrangement to convex sets in Euclidean space, and rearrangement to half-spaces in Gaussian space. 

\begin{prop} \label{prop: BMI rearrangement example}
When $(M,m,\mu) = (N,\eta, \alpha) = (\mathbb{R}^d, m_t, dx)$, and $t = (t_1, \dots, t_n) \in \mathbb{R}^n$, defines a map $m_t$ by vector space operations,
\begin{align}\label{eq:vector space addition map}
    x = (x_1, \dots, x_n) \mapsto \sum_{i=1}^n t_i x_i,
\end{align}
then the $*K$ rearrangement, as in Section \ref{sec:Preliminaries}, for $K$ open, convex, and symmetric, satisfies a set theoretic rearrangement inequality.  If the $t_i$ are assumed positive, $*K$ satisfies a set theoretic rearrangement without symmetry if $0$ belongs to the closure of $K$.
\end{prop}

\begin{proof}Take $B_i = \hbox{sgn} (t_i) A_i$ so that $t_1 A_1 + \cdots + t_n A_n = |t_1| B_1 + \cdots + |t_n| B_n$.  Using the symmetry and convexity of $K$, and the definition of our rearrangement as a scaling of $K$, it follows that
\begin{align*}
    t_1 A_1^* + \cdots + t_n A_n^*  = \left( \sum_{i=1}^n |t_i| |A_i|^{\frac 1 d} \right) K
\end{align*}
and hence that the images of $m_t$ are totally ordered.  Brunn-Minkowski implies that
\begin{align*}
    ||t_1|B_1 + \cdots + |t_n| B_n| \geq ||t_1| B_1^* + \cdots + |t_n| B_n^*|,
\end{align*}
it follows that
\begin{align*}
    |t_1 A_1 + \cdots + t_n A_n| \geq |t_1 A_1^* + \cdots + A_n^*|.
\end{align*}
When $t_i$ are positive, the proof is similar and simpler.
\end{proof}

\begin{prop} \label{prop: Gaussian rearrangement to one d}
    When $(M,m, \mu)$ is a centered Gaussian measure on a Banach space $M$ and $m$ defined as $x=(x_1,\dots,x_n) \mapsto \sum_i t_i x_i$ for $t_i > 0$, $\sum_i t_i =1$, and $(N,\eta,\alpha)$ with $N = \mathbb{R}^d$, $\eta$ defined by $y \mapsto \sum_i t_i y_i$ and $\alpha = \gamma_d$ the half-space rearrangement from Proposition \ref{prop: half space rearrangement holds} yields a set theoretic rearrangement inequality.
\end{prop}
This is the content of the Borell-Ehrhard theorem, which we will discuss in more detail in Section \ref{subsec: The Gaussian case}.  Now let us generalize the geometric mean used in PLI. 

\begin{defn}
\normalfont
For $0 < T \leq \infty$, a function  
$\mathcal{M}:[0,T)^n \to [0,\infty]$ is 
{\it continuous coordinate increasing} when 
\begin{enumerate}
    \item 
$x, y  \in \mathbb{R}^n$ satisfying $x_i > y_i$ for all $i$, necessarily satisfy $\mathcal{M}(x) > \mathcal{M}(y)$
\item
$\mathcal{M}(x) = 0$ when $\prod_i x_i = 0$
\item
$\mathcal{M}(x) = \sup_{y<x} \mathcal{M}(y)$ with the convention that $\sup_{y<x} \mathcal{M}(y) = 0$ when $\{ y < x\} $ is empty.
\end{enumerate}
\end{defn}
By convention, in the case that $T$ is finite, we extend $\mathcal{M}$ to $[0,T]^n$ by $\mathcal{M}(x) = \sup_{y<x} \mathcal{M}(y)$.
It should also be assumed tacitly, all $\mathcal{M}$ that follow are defined to be zero on  $\{x: \prod_i x_i = 0\}$.
\subsection*{Examples}
\begin{enumerate}
    \item  \label{eq: p - means coordinate increasing func}
        For $t = (t_1, \dots, t_n)$ with $t_i>0$ and $p \in [-\infty,0) \cup (0, \infty]$ take for $u \in [0,\infty)^n$
        \begin{align} \label{eq: p -means}
            \mathcal{M}_p^t(u) = \left( t_1 u_1^p + \cdots + t_n u_n^p \right)^{\frac 1 p}.
        \end{align}
        with $M_{-\infty}^t(u) = \min_i u_i$ and $M_\infty^t(u) = \max_i u_i$
    \item 
        For $t = (t_1, \cdots, t_n)$ with $t_i >0$ and $u \in [0,\infty)^n$,
        \begin{align} \label{eq: geometric mean}
            \mathcal{M}_{0}^t(u) = \prod u_i^{t_i}.
        \end{align} 
        Note that in the case that $\sum_i t_i = 1$, $\mathcal M_0^t$ is the limiting case of the previous example.
    \item
        Define for $t_i > 0$ and $u \in (0,1)^n$,
        \begin{align*}
        \mathcal{M}^t_{\Phi}(u) = \Phi(t_1 \Phi^{-1}(u_1) + \cdots + t_n \Phi^{-1}(u_n))
        \end{align*}
\end{enumerate}

Now let us define the functional operation our set theoretic rearrangement inequalities may be generalized to.
\begin{defn} \normalfont \label{defn: square sup definition}
        For $\mathcal{M}$ a continuous coordinate increasing function, $f = \{ f_i \}_{i=1}^n$ with $f_i: M \to [0,T)$, and $m: M^n \to M$ define
        \begin{align*} 
            \square_{\mathcal{M},m} f (z) \coloneqq \sup_{m(x)=z} \mathcal{M}(f_1(x_1), \dots, f_n(x_n)).
        \end{align*}
        Let us further denote for a rearrangement $*$ satisfying a set theoretic rearrangement inequality, $f_* = \{f_i^* \}_{i=1}^n$, so that
        \begin{align*}
            \square_{\mathcal{M}, \eta} f_*(w) = \sup_{\eta(y)=w} \mathcal{M}(f^*_1(y_1), \dots, f^*_n(y_n)).
        \end{align*} 
        When there is no risk of ambiguity we will suppress the notation for the mapping $m$ and write $\square_{\mathcal{M}} f$ in place of $\square_{\mathcal{M},m} f$.
\end{defn}

Notice that Theorem \ref{thm: PLI rearrangement} was the case that $m(x,y) = \eta(x,y) = (1-t)x+ty$ and $\mathcal{M}$ taken to be the geometric mean as in \eqref{eq: geometric mean}.   

\begin{thm} \label{thm: Main Theorem}
   A set theoretic rearrangement inequality,
    \begin{align*}
        \mu(m(A_1, \dots, A_n)) \geq \alpha(\eta(A_1^*, \dots, A_n^*))
    \end{align*}
    can be extended to functions in the sense that for $f= \{f_i\}_{i=1}^n$, with $f_i$ Borel measurable from $M$ to $[0,\infty)$, $\mathcal{M}$ a continuous coordinate increasing function, and a non-negative non-decreasing $\psi$,
    \begin{align*}
        \int \psi(\square_{{\mathcal{M},m}} f) d \mu \geq \int \psi(\square_{\mathcal{M},\eta} f_*) d\alpha.
    \end{align*}
\end{thm}

\begin{proof}
For $\lambda > 0$, write
\begin{align*}
   S_{\mathcal{M}}(\lambda) = \{ q \in \mathbb{Q}_+^n : \mathcal{M}(q) > \lambda \}.
\end{align*}
    We will prove $\mu( \square_{\mathcal{M}} f > \lambda ) \geq \alpha( \square_{\mathcal{M}} f_* > \lambda)$.  First observe that by arguments similar to the proof of Theorem \ref{thm: PLI rearrangement}
    \begin{align} \label{eq: description of superlevel sets as countable union}
        \{ \square_{\mathcal{M}} f > \lambda \} = \bigcup_{q \in S_{\mathcal{M}}(\lambda)} m(\{f_1> q_1\}, \dots, \{f_n > q_n \}).
    \end{align}
    Indeed, suppose $\square_{\mathcal{M}} f(z) > \lambda$.  This implies the existence of some $x$ such that $m(x) = z$ and $\mathcal{M}(f_1(x_1), \dots, f_n(x_n)) > \lambda$.  By the continuity of $\mathcal{M}$ there exists $q \in S_{\mathcal{M}}(\lambda)$ such that $\mathcal{M}(q_1, \dots, q_n) > \lambda$ and $f(x_i) > q_i$.  The opposite direction is immediate.  Observe that by our measurability assumptions on $m$ and \eqref{eq: description of superlevel sets as countable union}, the superlevel sets of $\square_{\mathcal{M},m} f$ are universally measurable.  Since $\psi$ is necessarily Borel measurable by its monotonicity, its composition with $\square_{\mathcal{M},m} f $ is indeed universally measurable.
    Analogously (note that $f_i^*$ are Borel measurable, by lower semi-continuity),
    \begin{align} \label{eq:  ***description of superlevel sets as countable union****}
              \{ \square_{\mathcal{M}} f_* > \lambda \} = \bigcup_{q \in S_{\mathcal{M}}(\lambda)} \eta(\{f_1^*> q_1\}, \dots, \{f_n^* > q_n \}).
    \end{align}
    This gives
    \begin{align*}
       \mu \{ \square_{\mathcal{M}} f > \lambda \} 
            &= 
                \mu \left(\bigcup_{q \in S_{\mathcal{M}}(\lambda)} m(\{f_1> q_1\}, \dots, \{f_n > q_n \}) \right).
               \\
            &\geq
                \sup_{q \in S_{\mathcal{M}}(\lambda)} \mu(m(\{f_1> q_1\}, \dots, \{f_n > q_n \}))
                    \\
            &\geq
                \sup_{q \in S_{\mathcal{M}}(\lambda)} \alpha(\eta(\{f_1> q_1\}^*, \dots, \{f_n > q_n \}^*))
                    \\
            &=
                \alpha \left(\bigcup_{q \in S_{\mathcal{M}}(\lambda)} \eta(\{f_1^*> q_1\}, \dots, \{f_n^* > q_n \}) \right)
                    \\
            &=
                \alpha \{ \square_{\mathcal{M}} f_* > \lambda \}
       \end{align*}
       where the first inequality is obvious, the second is by the assumed set theoretic rearrangment inequality, and the following equality is by the assumption of total orderedness.  The last equality is the from \eqref{eq:  ***description of superlevel sets as countable union****}.
    \end{proof}

\section{Applications} \label{sec: Applications}

\subsection{Borell-Brascamp-Lieb type inequalities}
In the case that $\lambda \in (0,1)$ and $-\infty \leq p \leq \infty$, we recall from example \eqref{eq: p - means coordinate increasing func} the following continuous coordinate increasing function,
\begin{align}
    \mathcal{M}(u,v) = \mathcal{M}_p^\lambda(u,v) = \begin{cases} 
      ((1-\lambda) u^p + \lambda v^p )^{\frac 1 p} & \text{if } uv \neq 0 \\
   0       & \text{if } uv =0.
  \end{cases}
\end{align}
 The Borell-Brascamp-Lieb inequality, generalizes the PLI with the understanding that $\mathcal{M}_0^\lambda (u,v) = u^{1-\lambda} v ^\lambda$.  Note that $\mathcal{M}_{\infty}^\lambda(u,v) = \max \{ u,v\}$ and $\mathcal{M}_{-\infty}^\lambda(u,v) = \min \{ u,v\}$ as defined in equation \eqref{eq: p -means}.   If we define $f \square_{\mathcal{M}_p^\lambda} g $ using  $m(x,y) = (1-\lambda) x + \lambda y$ as in Definition \ref{defn: square sup definition} we can state the inequality as the following.
\begin{thm}[Borell-Brascamp-Lieb \cite{Bor75a,BL76a}] \label{thm:Borell-Brascamp-Lieb}
    For $\lambda \in (0,1)$ and Borel functions $f,g: \mathbb{R}^n \to [0,\infty)$,
    \begin{align*}
        \int f \square_{\mathcal{M}_p^\lambda} g(x) \hspace{1mm}  dx \geq \mathcal{M}_{p/(np+1)}^\lambda \left( \int f(x) dx, \int g(x)  dx \right)
    \end{align*}
    when $p \geq -1/n$.
\end{thm}
We present the following sharpening.
\begin{thm}
 For Borel functions $f,g : \mathbb{R}^n \to [0,\infty)$ and $*$ a rearrangement to a convex set,
        \begin{align*}
            \int  f \square_{\mathcal{M}_p^\lambda} g(x) \hspace{1mm} dx 
                &\geq
                    \int f^* \square_{\mathcal{M}_p^\lambda} g^*(x) \hspace{1mm} dx
                        \\
                &\geq
                    \mathcal{M}_{p/(np+1)}^\lambda \left( \int f(x) dx, \int g(x)  dx \right)
        \end{align*}
        when $p \geq - 1/n$.
\end{thm}

\begin{proof}
As described in Proposition \ref{prop: BMI rearrangement example} the Brunn Minkowski inequality shows that the usual Lebesgue measure with the map $(x,y) \mapsto (1-\lambda) x + ty)$ satisfy a set theoretic rearrangement inequality.  The result then follows from Theorem \ref{thm: Main Theorem}.
\end{proof}

\subsection{The Gaussian case} \label{subsec: The Gaussian case}
For simplicity we restrict ourselves to the $\mathbb{R}^d$ case and 
employ the rearrangement $*$ from the Gaussian measure space $(\mathbb{R}^d, \gamma_d)$ to $(\mathbb{R}, \gamma_1 )$, by
\begin{align*}
	A^{*} = \{ x \in \mathbb{R} : x < t \}
\end{align*}
where $t=\Phi^{-1}(\gamma_d(A))$ is chosen to satisfy
$
	\gamma_d(A) = \gamma(A^{*}).
$
A functional half-space rearrangement by
\begin{align*}
	f^{*}(x) = \int_0^\infty \mathbbm{1}_{\{ f >t \}^{*}}(x) dt.
\end{align*}

The Borell-Ehrhard's inequality \cite{Bor08, Ehr83} is usually stated as the assertion that $t \in (0,1)$, $A, B$ Borel in $\mathbb{R}^d$ imply
	    \begin{align*}
	        \gamma_d((1-t)A + t B) \geq \Phi( (1-t) \Phi^{-1}(\mu(A)) + t \Phi^{-1}(\mu(B)) ).
	    \end{align*}
It can be equivalently formulated in our terminology and notation .
\begin{thm}[Borell, Ehrhard \cite{Bor08,Ehr83}] \label{thm: Borell Ehrhard}
For $t \in (0,1)$, $m(x,y) = (1-t)x + ty$, $\eta(u,v) = (1-t) u + tv$, and $*$ our halfspace rearrangement from $(\mathbb{R}^d,\gamma_d)$ to $(\mathbb{R},\gamma)$, satisfy a the set theoretic rearrangement inequality, explicitly for Borel $A$ and $B$
\begin{align*}
    \gamma_d ((1-t)A + t B) \geq \gamma((1-t) A^* + t B^*).
\end{align*}

\end{thm}
We will extend Theorem \ref{thm: Borell Ehrhard} to a functional inequality by Theorem \ref{thm: Main Theorem}.  However, it should be mentioned that the semigroup proof of Borell actually gave a functional inequality already. The argument was streamlined by Barthe and Huet and it is their generalization below that we will sharpen.
\begin{thm}[Barthe, Huet \cite{BH09}] \label{thm:Barthe}
Fix a set $I \subseteq \{1, 2, \dots, n\}$ and positive numbers $\lambda_1, \dots, \lambda_n$ satisfying 	$\sum \lambda_i \geq 1$ and
$
\lambda_j - \sum_{i \neq j} \lambda_i \leq 1
$
for $j \notin I$.
Then for Borel $f_1, \dots, f_n$ from $\mathbb{R}^d$ to $[0,1]$ such that $\Phi^{-1} \circ f_i$ is concave for $i \in I$, and a Borel
$h$ satisfying
$h(\sum_i \lambda_i x_i) \geq \Phi(\sum_i \lambda_i \Phi^{-1}(f_i(x_i)))$, then
	\begin{equation*}
		 \int h d \gamma_d  \geq \Phi \left( \lambda_1 \Phi^{-1}\left(\int f_1 d \gamma_d \right) + \cdots + \lambda_n \Phi^{-1} \left(\int f_n d \gamma_d \right) \right).
	\end{equation*}
\end{thm}

A consequence of Theorem \ref{thm:Barthe} (and actually proven equivalent to Theorem \ref{thm:Barthe} in the same paper) is the following.
\begin{cor} \label{cor: Set theoretic Barthe-Huet}
Fix a set $I \subseteq \{1, 2, \dots, n\}$ and set of positive numbers $\lambda_1, \dots, \lambda_n$ satisfying 	$\sum \lambda_i \geq 1$ and
$
\lambda_j - \sum_{i \neq j} \lambda_i \leq 1
$
for $j \notin I$.
Then for Borel $A_j$, 
\begin{align*}
\gamma_d( \lambda_1 A_1 + \cdots + \lambda_n A_n) 
    &\geq 
        \Phi(\lambda_1 \Phi^{-1}(\gamma_d(A_1)) + \cdots + \lambda_n \Phi^{-1}(\gamma_d (A_n)))
            \\
    &=
        \gamma(\lambda_1 A_1^* + \cdots + \lambda_n A_n^*)
\end{align*}
holds, provided $A_i$ are convex when $i \in I$.
\end{cor}

Strictly speaking, unless $I$ is empty, the half-line rearrangement does not yield a set theoretic rearrangement inequality with the maps $m_\lambda(x) = \lambda_1 x_1 + \cdots + \lambda_n x_n$ and $\eta_\lambda(y) = \lambda_1 y_1 + \cdots + \lambda_n y_n$.  
However the proof of Theorem \ref{thm: Main Theorem} can be adapted to achieve the following refinement of Barthe-Huet.

\begin{thm} \label{thm:Barthe extension}
	For Borel $f_1, \dots, f_n$ from $\mathbb{R}^d$ to $[0,1]$ such that $\Phi^{-1} \circ f_i$ is concave for $i \in I$ and 
	\begin{align*}
		\int \square_{\mathcal{M}_\Phi^\lambda} f d \gamma_d  
			&\geq 
				\int \square_{\mathcal{M}_\Phi^\lambda} f_* d \gamma
					\\
			&\geq
				\mathcal{M}_\Phi^\lambda \left(\int f_1^* d \gamma, \dots, \int f_n^* d \gamma \right)
					\\
			&=
				\mathcal{M}_\Phi^\lambda \left(\int f_1 d \gamma, \dots, \int f_n d \gamma \right).
	\end{align*}
\end{thm}

\begin{proof}
Once it is observed that $\Phi^{-1} \circ f_i$ concave ensures $\{ f_i > q_i\}$ is a convex set, so that one can apply Corollary \ref{cor: Set theoretic Barthe-Huet}, the first inequality can be derived following the proof of Theorem \ref{thm: Main Theorem}.  The equality is immediate as well, following from our definition of rearrangement.  Thus to prove the result we need only justify the second inequality, which follows from Theorem \ref{thm:Barthe} once we know that the concavity of $\Phi^{-1} \circ f_i$ implies the concavity of $\Phi^{-1} \circ f_i^*$ as well.  For this, we prove a general result below.
\end{proof}

\begin{defn} \normalfont
    For a fixed $t \in (0,1)$ and a convex set $K$ we will call $f: K \to \mathbb{R}$, {\it $\Psi_t$-concave} when there exists a continuous coordinate increasing function $\Psi_t$ 
    such that
        \begin{align*}
            f((1-t)x_1 + t x_2) \geq \Psi_t(f(x_1),f(x_2)).
        \end{align*}
\end{defn}

Notice that the concavity of $\Phi^{-1} \circ f$ is equivalent to the statement that $f$ is $\Psi_t$-concave with $\Psi_t(u_1,u_2) = \mathcal{M}_\Phi^t(u_1, u_2) = \Phi((1-t) \Phi^{-1}(u_1) + t \Phi^{-1}(u_2))$ for $t \in (0,1)$.
\begin{prop} \label{prop: rearrangement preserves concavity}
	Suppose that $f,g,h$ are Borel functions on a space $(M,\mu)$ satisfying
	\begin{align} \label{eq: Psi-concave three functions}
	    h((1-t)x+ty) \geq \Psi_t(f(x),g(y))
	\end{align}
	for $x,y \in M$,
	 and that $*$ is a rearrangement from $(M,\mu)$ to a space $(N,\alpha)$ satisfying 
	\begin{align}  \label{eq: Brunn-Minkowski-esque assumption}
	    \mu((1-t) A + t B) \geq \alpha((1-t) A^* + t B^*).
	\end{align}
	Additionally assume that the space of rearranged sets has a total ordering that respects Minkowski summation in the sense that $(1-t)A^*+ t B^*$ and $C^*$ satisfy either
	\begin{align} \label{eq: Minkowski summation ordering}
	    (1-t) A^* + t B^* \subseteq C^* \mbox{ or } (1-t) A^* + t B^* \supseteq C^*
	\end{align}
then 
    \begin{align}
        h^*((1-t)x + t y) \geq \Psi_t(f^*(x), g^*(y))
    \end{align}
    holds for $x,y \in N$.
\end{prop}

Note that Theorem \ref{thm:Barthe extension} follows from the proposition by taking $f = g =h$ and $\Psi_t = \mathcal{M}_\Phi^t$.  Indeed, since the half-line rearrangement satisfies \eqref{eq: Minkowski summation ordering}, as half-lines are stable under convex combination gives that $f^*$ to be $\mathcal{M}_\Phi^t$-concave if $f$ is. 

In analyzing the proof of Theorem \ref{thm:Barthe extension}, it presents an apparent loosening of the hypothesis requiring only that $f_i$ is quasi-concave and $f_i^*$ is $\mathcal{M}_\Phi^t$-concave.

\begin{proof}
Observe that inequality \eqref{eq: Psi-concave three functions} can be equivalently stated as $\lambda_i \in \mathbb{R}$ implies
\begin{align} \label{eq: set theoretic concavity}
	(1-t) \{ f > \lambda_1\} + t \{ g > \lambda_2\} \subseteq \{ h > \Psi_t(\lambda_1, \lambda_2) \}.
\end{align}
which can be easily verified using our assumptions of continuity and monotonicity.  Indeed, if \eqref{eq: Psi-concave three functions} holds, then for $z = (1-t)x + ty$ for $x \in \{ f > \lambda_1\}$ and $y \in \{ g > \lambda_2 \}$ we have $h(z) \geq \Psi_t(f(x),g(y)) > \Psi_t(\lambda_1, \lambda_2)$.  For the converse, given $x,y$ take $\lambda_1 < f(x)$ and $\lambda_2 < g(y)$, then $z = (1-t)x + t y \in (1-t) \{ f > \lambda_1\} + t \{ g > \lambda_2\}$.  By \eqref{eq: set theoretic concavity}, $h(z) > \Psi_t(f(x),g(y))$, and by the continuity assumption on $\Psi_t$, $ \Psi_t(f(x),g(y)) = \sup_\lambda \Psi_t(\lambda_1, \lambda_2) \leq h(z)$.
	Thus we will prove $(1-t) \{ f^* > \lambda_1\} + t \{ g^* > \lambda_2\} \subseteq \{ h^* > \Psi_t(\lambda_1, \lambda_2) \}$, or equivalently 
	\begin{align*}
		(1-t) \{ f > \lambda_1\}^* + t \{ g > \lambda_2\}^* \subseteq \{ h > \Psi_t(\lambda_1, \lambda_2) \}^*.
	\end{align*}
By \eqref{eq: Minkowski summation ordering}, it is enough to show
	\begin{align*}
		\alpha((1-t) \{ f > \lambda_1\}^* + t \{ g > \lambda_2\}^*) \leq \alpha(\{ h > \Psi_t(\lambda_1, \lambda_2) \}^*).
	\end{align*}
	
	By our assumptions \eqref{eq: Brunn-Minkowski-esque assumption} and \eqref{eq: set theoretic concavity},
	\begin{align*}
		\alpha((1-t) \{ f > \lambda_1\}^* + t \{ g > \lambda_2\}^*) &\leq \mu((1-t) \{ f > \lambda_1\} + t \{ g > \lambda_2\})
		    \\
		        &\leq
		            \mu(\{ h > \Psi_t(\lambda_1, \lambda_2)\}).
	\end{align*}
	Our result follows since
	\begin{align*}
		\mu(\{ h > \Psi_t(\lambda_1, \lambda_2) \}) = \alpha(\{ h > \Psi_t(\lambda_1, \lambda_2) \}^*).
	\end{align*}
\end{proof}

Let us also point out the corollary obtained by taking $f = g= h$, as it is of interest independent of the application to Theorem \ref{thm:Barthe extension}.

\begin{cor}
	If $f:\mathbb{R}^d \to [0,\infty)$ is $\Psi_t$-concave, and $*$ implies $f^*$ is as well.
\end{cor}
It follows immediately that the class of $d$-dimensional $s$-concave measures is stable under (convex set) rearrangement, see \cite{BM15, Bor74} for background, and \cite{li2019further,li2018further} for recent connections connections between $s$-concave measures and information theory.

Observe that Proposition \ref{prop: rearrangement preserves concavity} gives another proof of Theorem \ref{thm: PLI rearrangement}\eqref{eq: qualitative PLI}.  Indeed, since\\ $f\square g ((1-t)x+ty) \geq f^{1-t}(x) g^t(y)$ holds for all $x,y$, $(f\square g)^*((1-t)x+ty)) \geq (f^*)^{1-t}(x)(g^*)^t(y)$ holds as well.  This implies $(f \square g)^* \geq f^* \square g^* $ and hence $\int f \square g = \int (f\square g)^* \geq \int f^* \square g^*$.

\subsection{Polar Pr\'ekopa-Leindler}
For fixed $t,\lambda \in (0,1)$, define $\mathcal{M}: [0,\infty)^2 \to [0,\infty)$ by
\begin{align*}
    \mathcal{M}(u,v) 
        &=
            \min \left\{ u^{\frac{1-t}{1 - \lambda}}, v^{\frac t \lambda} \right\},
\end{align*}
and for $x,y \in \mathbb{R}^d$ define $m(x,y) = (1-t)x+ty$ so that
\[
    f \square_{\mathcal{M}} g (z) = \sup_{m(x,y) = z} \min \left\{ f(x)^{\frac{1-t}{1 - \lambda}}, g(y)^{\frac t \lambda} \right\}.
\] 
We can state the recent polar  analog of Pr\'ekopa-Leindler due to Artstein-Avidan, Florentin, and Segal.
\begin{thm}[Artstein-Avidan, Florentin, Segal \cite{artstein2017polar}] \label{thm: AFS Polar Prekopa}
    For $f,g: \mathbb{R}^d \to [0,\infty)$ Borel, and $\mu$ log-concave
    \begin{align*}
        \int f \square_{\mathcal{M}} g(x) d\mu(x)  \geq \mathcal{M}_{-1}^\lambda \left( \int f(x) d\mu(x), \int g(x) d\mu(x) \right).
    \end{align*}
\end{thm}

In the case that $\mu$ is Lebesgue (with $*$ rearrangement to a convex set) or Gaussian (with $*$ rearrangement to a half-space), and $\eta(x,y) = (1-t) x + ty$ this can be sharpened to the following.
\begin{thm}
       For $f,g: \mathbb{R}^d \to [0,\infty)$ Borel, and $\mu$ either Gaussian, with $*$ the half space rearrangement, or Lebesgue with $*$ a convex set rearrangement, then
    \begin{align*}
        \int f \square_{\mathcal{M}} g d\mu  
            &\geq 
                \int f^* \square_{\mathcal{M}} g^* d\mu
                    \\
            &\geq
                \mathcal{M}_{-1}^\lambda \left( \int f d\mu, \int g d\mu \right).
    \end{align*}
\end{thm}

\begin{proof}
As we have seen, the map $(x,y) \mapsto (1-t)x + t y$ satisfies a set theoretic rearrangement inequality by Brunn-Minkowski with respect to Lebesgue measure and rearrangement to a convex set, and by Borell-Ehrhard with respect to Gaussian measure and rearrangement to a halfspace.  The map $\mathcal{M}(u,v) = \min \{ u ^{\frac{1-t}{1-\lambda}}, v^{\frac t \lambda} \}$ is clearly continuous and coordinate increasing for $\lambda, t \in (0,1)$.  Thus in both cases, Gaussian and Lebesgue, we can invoke Theorem \ref{thm: Main Theorem} to obtain the first inequality.  The second inequality is obtained from the application of Theorem \ref{thm: AFS Polar Prekopa} to $f^*$ and $g^*$, and the equimeasurability of rearrangements.  
\end{proof}

\section{Gaussian log-Sobolev inequality}\label{sec: Gaussian log-Sobolev}
For a probability measure $\mu$ define the entropy functional\footnote{Note that when $f = \frac{d\nu}{d\mu}$ is the density function of a probability measure $\nu$ with respect to $\mu$, $H_\mu(f)$ is the Kullback-Liebler divergence $D( \nu || \mu)$ or relative entropy \cite{CT91:book}.} for a non-negative $f$ by
\begin{equation*}
    H_\mu(f) = \int f \log  f d\mu - \int f d\mu \log \int f d \mu.
\end{equation*}

One formulation of the Gaussian log-Sobolev inequality is the following.
\begin{thm}[Gaussian log-Sobolev]
For positive smooth $f$, 
\begin{align*}
    H_{\gamma_d}(f) \leq \frac 1 2 \int \frac{| \nabla f |^2}{f} d \gamma_d.
\end{align*}
\end{thm}
In this form the inequality is due to Gross \cite{Gro75}. Carlen \cite{Car91} showed it to be equivalent to the earlier information theoretic Blachman-Stam inequality \cite{Bla65,Sta59}. The Gaussian log-Sobolev inequality was shown to be a consequence of a strengthened PLI for strongly log-concave measures by Bobkov-Ledoux \cite{BL00}, and it is this perspective that we now develop to motivate the main result of this section, a rearrangement sharpening of an integrated Gaussian log-Sobolev inequality.   In this direction, let us recall that the PLI can be easily extended to the log-concave case.  
\begin{thm}[Log-concave PLI] \label{thm: logconcave PLI}
    For measure $\mu$ with density $\phi$ satisfying 
\begin{align*}
    \phi((1-t)x+ty) \geq \phi^{1-t}(x) \phi^t(y),
\end{align*}
the inequality for non-negative functions $u,v,w$
\begin{align*}
    u((1-t)x+ty) \geq v^{1-t}(x)w^t(y)
\end{align*}
implies
\begin{align} \label{eq: logconcave PLI conclusion}
    \int u d \mu  \geq \left( \int v d \mu \right)^{1-t} \left( \int w d \mu \right)^t.
\end{align}
\end{thm}

\begin{proof}
Observing that the functions $\tilde u(z) = u(z)\phi(z)$, $\tilde v (z) = v(z) \phi(z)$, and $\tilde w(z) = w(z) \phi(z)$ satisfy 
\begin{align*}
    \tilde u ((1-t)x+ty) \geq \tilde v^{1-t}(x) \tilde w^{t}(y)
\end{align*}
so that applying the ordinary PLI, we have
\begin{align*}
    \int \tilde u(z) dz \geq \left( \int \tilde v(z) dz \right)^{1-t} \left( \int \tilde w(z) dz \right)^{t},
\end{align*}
which is exactly \eqref{eq: logconcave PLI conclusion}.
\end{proof}

The log-concave case corresponds to the case when the measure is given by a density corresponding to a convex potential, that is $\phi(x) = e^{-V(x)}$ when $V$ is convex.  For the Gaussian measure something stronger is true, $V$ in this case satisfies 
\begin{align} \label{eq: strong convexity}
    V((1-t)x+ty) \leq (1-t)V(x) + t V(y) - t(1-t) |x-y|^2/2.
\end{align}
Note that in the case that $V$ is smooth, log-concavity is exactly $V'' \geq 0_d$ in the sense of positive semi-definite matrices, while \eqref{eq: strong convexity} is $V'' \geq I_d$.  Under these assumptions, Theorem \ref{thm: logconcave PLI} admits the following strengthening.

\begin{thm}[Curved Pr\'ekopa-Leindler] \label{thm: Gaussian Prekopa-Leindler}
For $t \in (0,1)$, $\mu$ strongly log-concave in the sense of \eqref{eq: strong convexity}, and $u,v,w: \mathbb{R}^d \to [0,\infty)$ satisfying
\begin{align*}
    u((1-t)x + ty) \geq e^{-t(1-t) |x-y|^2/2} v^{1-t}(x)w^t(y),
\end{align*}
for all $x,y \in \mathbb{R}^d$, then
\begin{align*}
    \int u d \mu \geq \left(\int v \ d \mu \right)^{1-t} \left( \int w d \mu \right)^{t}.
\end{align*}
\end{thm}

\begin{proof}
The proof follows again from applying the Euclidean PLI to $\tilde u(z) = u(z)\phi(z)$, $\tilde v (z) = v(z) \phi(z)$.
\end{proof}


Following arguments of Bobkov-Ledoux \cite{BL00} we pursue a specialization of Theorem \ref{thm: Gaussian Prekopa-Leindler} to a single function, revealing a log-Sobolev inequality as a consequence of a strengthened PLI.  For a fixed $t \in (0,1)$, and a strongly log-concave probability measure $\mu$, and $f$, take $w = f^{\frac 1 t}$, $v =1$, then for any $u$, satisfying
\begin{align*}
    u((1-t)x + t y ) \geq e^{-t(1-t) |x-y|^2/2} f(y)
\end{align*}
we have from Theorem \ref{thm: Gaussian Prekopa-Leindler}
\begin{align*}
    \int u \hspace{1mm} d\mu \geq \left( \int f^{\frac 1 t} d \mu \right)^t.
\end{align*}
With the interest of determining the optimal such $u$ achievable through the methods of PLI, it is natural to consider
\begin{align*}
    u(z) = \sup_{ \{(x,y):(1-t)x + ty = z \}}  e^{-t(1-t) |x-y|^2/2} f(y).
\end{align*}
Writing $\lambda = \frac{1-t}{t}$, note that the constraint on $x,y$ is equivalent to $y = z + \lambda (z-x)$, so that the  $u(z)$ above can be expressed as $Q_\lambda f(z)$ in the following definition.
\begin{defn}
For $\lambda \in (0,\infty)$ and $f$ non-negative and Borel measurable, define
\begin{align*}
    Q_\lambda f(z) &= \sup_w f(z + \lambda w) e^{-\lambda |w|^2/2} \\
    &=
    \sup_w f(z + w) e^{- |w|^2/2 \lambda}.
\end{align*}
\end{defn}
Writing $\| f \|_p = \left( \int |f|^p d \mu \right)^{\frac 1 p}$ we can collect the above as the following.

\begin{thm}[Integrated log-Sobolev] \label{thm: Integrated log-Sobolev}
    For $\mu$ a strongly log-concave probability measure, $\lambda \in (0,\infty)$ and $f$ non-negative and Borel measurable,
    \[
        \| Q_\lambda f \|_1 \geq \| f \|_{1+\lambda}.
    \]
\end{thm}

The log-Sobolev inequality for strongly log-concave probability measures can be recovered as a corollary.

\begin{cor}[Log-Sobolev inequality]
For $\mu$ strongly log-concave probability measure, and $f$ a positive smooth function
\[
    H_\mu(f) \leq \frac 1 2 \int \frac{| \nabla f |^2}{f} d\mu 
\]
\end{cor}

A proof is given in \cite{BL00}, where the expressions are given in terms of $f^2$ rather than $f$.  It follows as a limiting case of Theorem \ref{thm: Integrated log-Sobolev} with $\lambda \to 0$.  
\begin{proof}[Sketch of proof]For smooth positive functions constant outside of a compact set, one observes that equality holds when $\lambda = 0$. Then the Taylor series expansion, 
\[
    \| f \|_{1+\lambda} = \|f\|_1 + \lambda H_\mu(f) + o(\lambda)
\]
and a derived inequality
\[
    \| Q_\lambda f \|_1 \leq \|f\|_1 + \frac{\lambda}{2} \int \frac{| \nabla f |^2}{ f} d\mu + o(\lambda)
\]
deliver the conclusion. A limiting argument gives the result for general functions.
\end{proof}

Now let us specialize to the case that $\mu = \gamma_d$ a standard Gaussian, and $*$ denote the half-space rearrangement of a set under $\gamma_d$ as in Proposition \ref{prop: Gaussian rearrangement to one d} and we can state our main result of the section.

\begin{thm} \label{thm: infimum convolution like inequality}
    For non-negative Borel $f$ and $\lambda,s > 0$,
    \begin{align*}
        \gamma_d ( \{ Q_\lambda f > s \} ) \geq \gamma( \{Q_\lambda f^*  > s \} )
    \end{align*}
    where $f^*$ is the Gaussian half-line rearrangement of $f$.
\end{thm}

It will be a consequence of the proof that $Q_\lambda f$ is universally measurable.

\begin{proof}
We first express $\{ Q_\lambda f > s \}$ as the union of simpler sets.  Denoting
\[
S = S(s,q_1,q_2) = \{ q = (q_1,q_2) \in \mathbb{Q}_+^2: q_1 q_2 > s \},
\] it is straight forward to verify
\begin{align} \label{eq: equivalence of infimum convolution super level sets}
    \{ Q_\lambda f > s \} = \bigcup_{q \in S} \left( \{x \in \mathbb{R}^d : f(x) > q_1 \} + \left\{y \in \mathbb{R}^d : |y| < \sqrt{ 2 \lambda \ln \frac 1 {q_2}  } \right\} \right) .
\end{align}
Indeed, for $z$ belonging to the union,  there exists rational $q_i$, and $x,y$ satisfying $f(x) > q_1$, $|y| < \sqrt{ 2 \lambda \ln \frac 1 {q_2} }$,  and $x+y = z$.  Taking $w = -x = y-z$,
\begin{align*}
    f(w)e^{-|w|^2/2\lambda} > q_1 q_2 > s,
\end{align*}
so that
$z \in \{ Q_\lambda f > s\}$.  Conversely if there exists a $w$ such that $f(z+w) e^{-|w|^2/ 2\lambda} > s$ then by continuity there exist rational $q_i$ satisfying $f(z+w) > q_1$, $e^{-|w|^2/2 \lambda} > q_2$, and $q_1 q_2 > s$.  Taking $x = z+w$ and $y= -w$ we see that $(q_1,q_2) \in S$ and
\begin{align*}
    z \in \{ f > q_1 \} + \left\{ |y| < \sqrt{ 2 \lambda \ln \frac 1 {q_2} } \right\}.
\end{align*}
Notice that this gives $\{Q_\lambda f > s \}$ as a countable union of Minkowski sums of analytic sets.  Since analytic sets are closed under such operations, $\{Q_\lambda f > s\}$ is an analytic set as well, and the universal measurability of $Q_t f$ follows.

 Applying the Gaussian isoperimetric inequality \cite{borell1975brunn, ST78}, which in our preferred formulation states that $\gamma_d( A + B_d) \geq \gamma(A^* + B_1)$ where $B_d$ and $B_1$ are origin symmetric Euclidean balls of equal radius (in $\mathbb{R}^d$ and $\mathbb{R}$ respectively), we have
\begin{align*}
    \gamma_d(\{ Q_\lambda f > s\})
        &=
            \gamma_d \left( \bigcup_{q \in S} \{ f > q_1 \} + \left\{w \in \mathbb{R}^d : |w| < \sqrt{ 2 \lambda \ln \frac 1 {q_2}  } \right\} \right) 
                \\
        &\geq
            \sup_{q \in S} \gamma_d \left(\{ f > q_1 \} + \left\{w \in \mathbb{R}^d : |w| < \sqrt{ 2 \lambda \ln \frac 1 {q_2}  } \right\} \right)
                \\
        &\geq
            \sup_{q \in S} \gamma \left(\{ f > q_1 \}^* + \left\{w \in \mathbb{R} : |w| < \sqrt{ 2 \lambda \ln \frac 1 {q_2}  } \right\} \right).
\end{align*}
But $ \{ f > q_1 \}^* = \{ f^* > q_1 \}$ is a half-line and hence the family of 
$\{f^* > q_1 \} +  \left\{|w| < \sqrt{ 2 \lambda \ln \frac 1 {q_2}  } \right\}$ indexed by $S(\lambda,q_1,q_2)$ is a family of totally ordered sets.  Thus,
\begin{align*}
    \sup_{q \in S} \gamma \left(\{ f > q_1 \}^* + \left\{|w| < \sqrt{ 2 \lambda \ln \frac 1 {q_2}  } \right\} \right) = 
        \gamma \left( \bigcup_{q \in S} \{ f^* > q_1 \} + \left\{|w| < \sqrt{ 2 \lambda \ln \frac 1 {q_2}  } \right\} \right).
\end{align*}
    Applying \eqref{eq: equivalence of infimum convolution super level sets} we have
\begin{align*}
    \gamma \left( \bigcup_{q \in S} \{ f^* > q_1 \} + \left\{|w| < \sqrt{ 2 \lambda \ln \frac 1 {q_2}  } \right\} \right)
        =
            \gamma(\{ Q_\lambda f^* > \lambda \}),
\end{align*}
and our theorem follows.
\end{proof}

We have as an immediate consequence, a sharpening of Theorem \ref{thm: Integrated log-Sobolev}.
\begin{cor}
    For $f$ non-negative and Borel, and norms taken with respect to $\gamma$,
    \[
        \int Q_\lambda f d \gamma \geq \int Q_\lambda f^* d \gamma \geq \|f^*\|_{1+\lambda}  = \|f\|_{1+\lambda}.
    \]
\end{cor}
\begin{proof}
    The first inequality is a consequence of Theorem \ref{thm: infimum convolution like inequality}, while the second is from Theorem \ref{thm: Integrated log-Sobolev}.
\end{proof}

We also direct the reader to the articles \cite{martin2009isoperimetry, martin2010pointwise} of Mart\'in and M. Milman,  whose work on symmetrization, isoperimetry, and log-Sobolev inequalities the author learned of during the revision of this paper.

\section{Barthe, Brascamp, Lieb and Rearrangement} \label{sec: Barthe Rearrangement Connection}
The Brascamp-Lieb inequality is the following.
\begin{thm}[Brascamp, Lieb \text{\cite{BL76b}}] \label{thm:Brascamp-Lieb}
    For natural numbers $n \leq m$, and $\{n_i\}_{i=1}^m$ with $n_i \leq n$ and $\{c_i\}_{i=1}^m$ a sequence of positive numbers such that $\sum_{i=1}^m c_i n_i = n$ then for surjective linear maps $B_i: \mathbb{R}^{n} \to \mathbb{R}^{n_i}$, with $\cap_i \ker(B_i) = 0$ and transposes denoted $B_i'$ satisfy the following,
    \begin{align*}
        \int_{\mathbb{R}^n} \prod_{i=1}^m f^{c_i}_i(B_i x) dx \leq C^{-1/2} \prod \left( \int_{\mathbb{R}^{n_i}} f_i \right)^{c_i}
    \end{align*}
    for $f_i :\mathbb{R}^{n_i} \to [0,\infty)$ integrable, and 
    \begin{align*}
        C = \inf \left\{ \frac{ \det( \sum_{i=1} c_i B_i' A_i B_i )}{ \prod \det^{c_i}{A_i} } : A_i \mbox{ positive definite} \right\}.
    \end{align*}
\end{thm}
The theorem enjoys a qualitative analog in the case that $n_i = d$, so that $n = md$ and $x \in \mathbb{R}^n$ can be expressed as $x = (x_1, \dots, x_m)$ for $x_j \in \mathbb{R}^d$ and $B_i$ are of the form 
\begin{align} \label{eq: BLL defintion of linear maps}
    B_i x = \sum_{j = 1}^m B_{ij} x_j
\end{align}
then the rearrangement theorem due to Brascamp-Lieb-Luttinger is what follows.
\begin{thm}[Brascamp, Lieb, Luttinger \cite{BLL74}] \label{thm:BLL} 
For $B_i$ satisfying \eqref{eq: BLL defintion of linear maps},
\begin{align*}
    \int_{\mathbb{R}^n} \prod_{i=1}^m f_i( B_i x) dx \leq \int_{\mathbb{R}^n} \prod_{i=1}^m f_i^*(B_i x) dx,
\end{align*}
where $*$ represents the spherically symmetric decreasing rearrangement.
\end{thm}
Notice that when Theorem \ref{thm:BLL} applies, it gives an intermediary inequality to Theorem \ref{thm:Brascamp-Lieb}.  Indeed since $(f^{c_i})^* = (f^*)^{c_i}$, applying Theorem \ref{thm:BLL} and then \ref{thm:Brascamp-Lieb} gives
\begin{align*}
    \int_{\mathbb{R}^n} \prod_{i=1}^m f^{c_i}(B_i x) dx 
        &\leq 
            \int_{\mathbb{R}^n} \prod_{i=1}^m (f^*)^{c_i}(B_i x) dx 
                \\
        &\leq
        C^{-1/2} \prod_{i=1}^m \left( \int_{\mathbb{R}^{n_i}} f \right)^{c_i}.
\end{align*}

Barthe gave the following reversal of Brascamp-Lieb, that serves as a dual inequality.
\begin{thm}[Barthe \cite{Bar98b}]
    For $n$, $m$, $\{n_i\}_{i=1}^m$, $\{c_i\}_{i=1}^m$, $B_i$, and $C$ as in Theorem \ref{thm:Brascamp-Lieb} then the inequality
    \begin{align*}
        C^{1/2} \prod_{i=1}^m \left( \int_{\mathbb{R}^{n_i}} f_i \right)^{c_i} \leq  \int_{\mathbb{R}^n} \sup \left\{  \prod_{i=1}^m f^{c_i}_i (y_i) : \sum_i c_i B_i' y_i = x  \right\}dx,
    \end{align*}
    holds for $f_i :\mathbb{R}^{n_i} \to [0,\infty)$ integrable.
\end{thm}

Taking $m=2$, $c_1 = (1-t), c_2 = t$ and $n_i = n$ and $B_i$ to be the identity map, yields $C=1$ and we recover the Prekopa-Liendler inequality.  We ask if further extensions of our work here exist.
\begin{ques} \label{ques: reverse BLL}
    Suppose that $B_i$ are of the form \eqref{eq: BLL defintion of linear maps}, and $f_i: \mathbb{R}^d \to [0,\infty)$, when is it true that
    \begin{align} \label{eq: reverse BLL}
        \int_{\mathbb{R}^n} \sup \left\{  \prod_{i=1}^m f_i(y_i) : \sum_i B_i' y_i = x  \right\}dx \geq \int_{\mathbb{R}^n} \sup \left\{  \prod_{i=1}^m f_i^*(y_i) : \sum_i B_i' y_i = x  \right\}dx
    \end{align}
    holds?
\end{ques}
The results presented here verify the inequality for general Borel $f_i$ in the case that $B_i$ are scalar multiples of the identity.   Note that in the case that $f_i  = \mathbbm{1}_{A_i}$, asks if the following generalization of BMI holds
\begin{equation} \label{eq: Feder and Zamir}
    \left|\sum_i B_i' A_i \right|_n \geq \left|\sum_i B_i' A_i^* \right|_n,
\end{equation}
where 
\begin{equation*}
    \sum_i B_i' A_i = \left\{ z = \sum_i B_i' x_i : x_i \in A_i \right\}.
\end{equation*}
In the case that $B_i' : \mathbb{R} \to \mathbb{R}^d$, inequality \eqref{eq: Feder and Zamir} was proven by Zamir and Feder \cite{ZF93}. 

\section{Acknowledgements}
This work was supported by NSF grants CMMI 1462862 and ECCS 1809194.
A portion of this work relevant to information theory was announced at 56th Annual Allerton Conference on Communication, Control, and Computing \cite{melbourne2018allerton}.
\bibliographystyle{plain}
\bibliography{bibibi}

\begin{thebibliography}{10}

\bibitem{artstein2017polar}
S.~Artstein-Avidan, D.~Florentin, and A.~Segal.
\newblock Polar {P}r\'ekopa--{L}eindler inequalities.
\newblock {\em arXiv preprint arXiv:1707.08732}, 2017.

\bibitem{Bar98b}
F.~Barthe.
\newblock On a reverse form of the {B}rascamp-{L}ieb inequality.
\newblock {\em Invent. Math.}, 134(2):335--361, 1998.

\bibitem{BH09}
F.~Barthe and N.~Huet.
\newblock On {G}aussian {B}runn--{M}inkowski inequalities.
\newblock {\em Studia Math.}, 191(3):283--304, 2009.

\bibitem{Bec75}
W.~Beckner.
\newblock Inequalities in {F}ourier analysis.
\newblock {\em Ann. of Math. (2)}, 102(1):159--182, 1975.

\bibitem{Bla65}
N.M. Blachman.
\newblock The convolution inequality for entropy powers.
\newblock {\em IEEE Trans. Information Theory}, IT-11:267--271, 1965.

\bibitem{BM16}
S.~Bobkov and A.~Marsiglietti.
\newblock Variants of the entropy power inequality.
\newblock {\em IEEE Transactions on Information Theory}, 63(12):7747--7752,
  2017.

\bibitem{BM15}
S.~Bobkov and J.~Melbourne.
\newblock Hyperbolic measures on infinite dimensional spaces.
\newblock {\em Probability Surveys}, 13:57--88, 2016.

\bibitem{BC14}
S.~G. Bobkov and G.~P. Chistyakov.
\newblock Bounds for the maximum of the density of the sum of independent
  random variables.
\newblock {\em Zap. Nauchn. Sem. S.-Peterburg. Otdel. Mat. Inst. Steklov.
  (POMI)}, 408(Veroyatnost i Statistika. 18):62--73, 324, 2012.

\bibitem{BC15:1}
S.~G. Bobkov and G.~P. Chistyakov.
\newblock Entropy power inequality for the {R\'enyi} entropy.
\newblock {\em IEEE Trans. Inform. Theory}, 61(2):708--714, February 2015.

\bibitem{BL00}
S.~G. Bobkov and M.~Ledoux.
\newblock From {B}runn-{M}inkowski to {B}rascamp-{L}ieb and to logarithmic
  {S}obolev inequalities.
\newblock {\em Geom. Funct. Anal.}, 10(5):1028--1052, 2000.

\bibitem{BL09}
S.~G. Bobkov and M.~Ledoux.
\newblock Weighted {P}oincar\'e-type inequalities for {C}auchy and other convex
  measures.
\newblock {\em Ann. Probab.}, 37(2):403--427, 2009.

\bibitem{Bor74}
C.~Borell.
\newblock Convex measures on locally convex spaces.
\newblock {\em Ark. Mat.}, 12:239--252, 1974.

\bibitem{borell1975brunn}
C~Borell.
\newblock The {B}runn-{M}inkowski inequality in {G}auss space.
\newblock {\em Inventiones mathematicae}, 30(2):207--216, 1975.

\bibitem{Bor75a}
C.~Borell.
\newblock Convex set functions in {$d$}-space.
\newblock {\em Period. Math. Hungar.}, 6(2):111--136, 1975.

\bibitem{Bor08}
C.~Borell.
\newblock Inequalities of the {B}runn-{M}inkowski type for {G}aussian measures.
\newblock {\em Probab. Theory Related Fields}, 140(1-2):195--205, 2008.

\bibitem{BL76b}
H.~J. Brascamp and E.~H. Lieb.
\newblock Best constants in {Y}oung's inequality, its converse, and its
  generalization to more than three functions.
\newblock {\em Advances in Math.}, 20(2):151--173, 1976.

\bibitem{BL76a}
H.~J. Brascamp and E.~H. Lieb.
\newblock On extensions of the {B}runn-{M}inkowski and {P}r\'ekopa-{L}eindler
  theorems, including inequalities for log concave functions, and with an
  application to the diffusion equation.
\newblock {\em J. Functional Analysis}, 22(4):366--389, 1976.

\bibitem{BLL74}
H.~J. Brascamp, E.~H. Lieb, and J.~M. Luttinger.
\newblock A general rearrangement inequality for multiple integrals.
\newblock {\em J. Functional Analysis}, 17:227--237, 1974.

\bibitem{Car91}
E.~A. Carlen.
\newblock Superadditivity of {F}isher's information and logarithmic {S}obolev
  inequalities.
\newblock {\em J. Funct. Anal.}, 101(1):194--211, 1991.

\bibitem{CC84}
M.~H.~M. Costa and T.~M. Cover.
\newblock On the similarity of the entropy power inequality and the
  {B}runn-{M}inkowski inequality.
\newblock {\em IEEE Trans. Inform. Theory}, 30(6):837--839, 1984.

\bibitem{CT91:book}
T.~M. Cover and J.~A. Thomas.
\newblock {\em Elements of Information Theory}.
\newblock J. Wiley, New York, 1991.

\bibitem{DCT91}
A.~Dembo, T.~M. Cover, and J.~A. Thomas.
\newblock Information-theoretic inequalities.
\newblock {\em IEEE Trans. Inform. Theory}, 37(6):1501--1518, 1991.

\bibitem{Ehr83}
A.~Ehrhard.
\newblock Sym\'etrisation dans l'espace de {G}auss.
\newblock {\em Math. Scand.}, 53(2):281--301, 1983.

\bibitem{Gro75}
L.~Gross.
\newblock Logarithmic {S}obolev inequalities.
\newblock {\em Amer. J. Math.}, 97(4):1061--1083, 1975.

\bibitem{kechris2012classical}
A.~Kechris.
\newblock {\em Classical descriptive set theory}, volume 156.
\newblock Springer Science \& Business Media, 2012.

\bibitem{Li17}
J.~Li.
\newblock R{\'e}nyi entropy power inequality and a reverse.
\newblock {\em Preprint, arXiv:1704.02634}, 2017.

\bibitem{LMM2019:isit:1}
J.~Li, A.~Marsiglietti, and J.~Melbourne.
\newblock Entropic central limit theorem for r\'enyi entropy.
\newblock In {\em 2018 IEEE International Symposium on Information Theory
  (ISIT)}. IEEE, 2019.

\bibitem{li2019further}
J.~Li, A.~Marsiglietti, and J.~Melbourne.
\newblock Further investigations of {R}\'enyi entropy power inequalities and an
  entropic characterization of s-concave densities.
\newblock {\em arXiv preprint arXiv:1901.10616}, 2019.

\bibitem{LMM2019:isit:2}
J.~Li, A.~Marsiglietti, and J.~Melbourne.
\newblock R\'enyi entropy power inequalities for $s$-concave densities.
\newblock In {\em 2019 IEEE International Symposium on Information Theory
  (ISIT)}. IEEE, 2019.

\bibitem{li2018further}
J.~Li and J.~Melbourne.
\newblock Further investigations of the maximum entropy of the sum of two
  dependent random variables.
\newblock In {\em 2018 IEEE International Symposium on Information Theory
  (ISIT)}, pages 1969--1972. IEEE, 2018.

\bibitem{Lie78}
E.~H. Lieb.
\newblock Proof of an entropy conjecture of {W}ehrl.
\newblock {\em Comm. Math. Phys.}, 62(1):35--41, 1978.

\bibitem{MMX17:1}
M.~Madiman, J.~Melbourne, and P.~Xu.
\newblock Forward and reverse entropy power inequalities in convex geometry.
\newblock {\em Convexity and Concentration}, pages 427--485, 2017.

\bibitem{MMX17:2}
M.~Madiman, J.~Melbourne, and P.~Xu.
\newblock Rogozin's convolution inequality for locally compact groups.
\newblock {\em Preprint, arXiv:1705.00642}, 2017.

\bibitem{marsiglietti2018entropy}
A.~Marsiglietti and J.~Melbourne.
\newblock On the entropy power inequality for the {R}{\'e}nyi entropy of order
  [0, 1].
\newblock {\em IEEE Transactions on Information Theory}, 2018.

\bibitem{marsiglietti2018renyi}
A.~Marsiglietti and J.~Melbourne.
\newblock A {R}{\'e}nyi entropy power inequality for log-concave vectors and
  parameters in [0, 1].
\newblock In {\em 2018 IEEE International Symposium on Information Theory
  (ISIT)}, pages 1964--1968. IEEE, 2018.

\bibitem{martin2009isoperimetry}
J.~Mart{\'\i}n and M.~Milman.
\newblock Isoperimetry and symmetrization for logarithmic sobolev inequalities.
\newblock {\em Journal of Functional Analysis}, 256(1):149--178, 2009.

\bibitem{martin2010pointwise}
J.~Mart{\'\i}n and M.~Milman.
\newblock Pointwise symmetrization inequalities for sobolev functions and
  applications.
\newblock {\em Advances in Mathematics}, 225(1):121--199, 2010.

\bibitem{melbourne2018allerton}
J.~Melbourne.
\newblock Rearrangements and information theoretic inequalities.
\newblock In {\em Communication, Control, and Computing (Allerton), 2012 50th
  Annual Allerton Conference on}. IEEE, 2018.

\bibitem{RS16}
E.~Ram and I.~Sason.
\newblock On r{\'e}nyi entropy power inequalities.
\newblock {\em IEEE Transactions on Information Theory}, 62(12):6800--6815,
  2016.

\bibitem{Ren61}
A.~R{\'e}nyi.
\newblock On measures of entropy and information.
\newblock In {\em Proc. 4th Berkeley Sympos. Math. Statist. and Prob., Vol. I},
  pages 547--561. Univ. California Press, Berkeley, Calif., 1961.

\bibitem{Sha48}
C.E. Shannon.
\newblock A mathematical theory of communication.
\newblock {\em Bell System Tech. J.}, 27:379--423, 623--656, 1948.

\bibitem{Sta59}
A.J. Stam.
\newblock Some inequalities satisfied by the quantities of information of
  {F}isher and {S}hannon.
\newblock {\em Information and Control}, 2:101--112, 1959.

\bibitem{Sta69}
R.~M. Starr.
\newblock Quasi-equilibria in markets with non-convex preferences.
\newblock {\em Econometrica}, 37(1):25--38, January 1969.

\bibitem{ST78}
V.N. Sudakov and B.S. Tsirel'son.
\newblock Extremal properties of half-spaces for spherically invariant
  measures.
\newblock {\em Zap. Nauch. Sem. L.O.M.I.}, 41:14--24, translated in J. Soviet
  Math. 9, 9--18 (1978) 1974.

\bibitem{WM14}
L.~Wang and M.~Madiman.
\newblock Beyond the entropy power inequality, via rearrangements.
\newblock {\em IEEE Trans. Inform. Theory}, 60(9):5116--5137, September 2014.

\bibitem{ZF93}
R.~Zamir and M.~Feder.
\newblock A generalization of the entropy power inequality with applications.
\newblock {\em IEEE Trans. Inform. Theory}, 39(5):1723--1728, 1993.

\end{thebibliography}

\end{document}